\documentclass[a4paper]{article}
\usepackage{tikz}
\usetikzlibrary{shapes.geometric}
\usetikzlibrary{calc}
\usepackage{amsmath}
\usepackage{mathtools}
\usepackage{amsthm}
\usepackage{multicol}

\usepackage{fullpage}
\usepackage{enumitem}
\setlist[itemize]{noitemsep}

\usepackage{natbib}
\newcommand{\cac}{olive}
\newcommand{\cbc}{teal}
\newcommand{\ccc}{orange}
\newcommand{\cdc}{purple}

\usepackage{mdwlist}
\renewenvironment{description}%
{
\begin{basedescript}{
\desclabelstyle{\nextlinelabel}
\desclabelwidth{2em}}}
{
\end{basedescript}
}

\usepackage{titlesec}
\titleformat*{\section}{\fontsize{10pt}{10pt}\bfseries\selectfont}

\newtheorem{theorem}{Theorem}

\begin{document}

\begin{center}
    
{ \bf \large
The chromatic number  of the associahedron: simple and logarithmic
}

\vspace{5pt}
Benjamin Aram Berendsohn \\
{\it Max Planck Institute for Informatics}

\vspace{5pt}
Jean Cardinal and
John Iacono \\ {\it Université libre de Bruxelles}

\vspace{5pt}
László Kozma \\{\it Technische Universität Dresden}\\%

\end{center}

\begin{abstract}
  Addario-Berry, Reed,  Scott, and Wood  (JoCG 2026) proved that the $n$-dimensional associahedron has chromatic number at most $22500\cdot\log_3 n+O(1)$. We present a much simpler proof that the chromatic number of the $n-1$-dimensional associahedron is at most $6 \log_2 n + 18$. Since this paper's appearance as a preliminary abstract in the Japan Conference on Discrete and Computational Geometry, Graphs, and Games held from September 7-10, 2026, the bound has been improved to $O(\log \log n)$ by Oum and Wood (arXiv September 30, 2026). It is hoped that the simple method here could give insights to further improvements.
\end{abstract}

\begin{figure}[h]
\begin{center}
\resizebox{10cm}{!}{\input{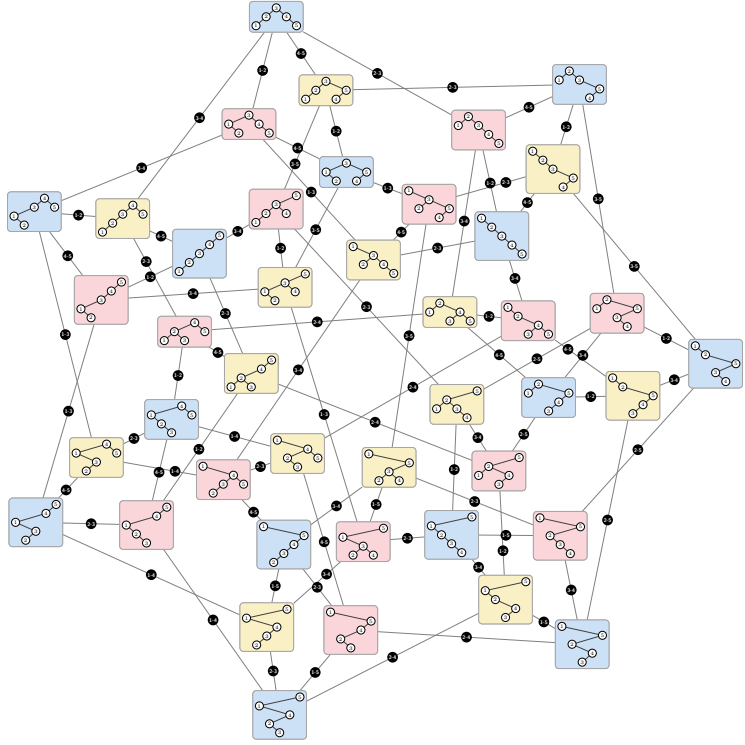}}
\end{center}
\caption{The graph of the $4$-dimensional associahedron: Nodes are all binary search trees on $\{1,2,3,4,5\}$. Edges connect those nodes representing trees that are a single rotation apart where the nodes of the rotation appear in the black circle on each edge. The graph is 3-colored. 
}
\end{figure}

\paragraph{Introduction.}
Associahedra are polytopes that appear in many areas of mathematics.
We refer to the survey of Pilaud, Santos, and Ziegler~\cite{MR4675114} for an overview of the developments around these polytopes in the past few decades.
We are interested in the graphs of associahedra.
For a given positive integer~$n$, the vertices of the $(n-1)$-dimensional associahedron are in bijection with triangulations of a convex $(n+2)$-gon, or, via planar duality, with rooted binary trees on $n$ (internal) nodes.
Hence for a given $n$, the number of vertices is $C_n = \frac{(2n)!}{(n+1)!n!}$, the $n$th Catalan number.
Edges of associahedra connect pairs of triangulations that differ by a single
\emph{diagonal-flip}, an operation that consists of 
replacing a diagonal of a quadrilateral
formed by two triangles sharing an edge by the other diagonal. Equivalently
--- and this is the representation that we will use in our proofs ---
edges connect pairs of vertices corresponding to two binary trees that differ by
a single rotation.

While being well-studied, the structure of the associahedron graph still remains
mysterious. It is known to have diameter $2n-6$ for sufficiently large $n$~\cite{MR4556611,MR3197650}, and
to be Hamiltonian (see for instance~\cite{MR3383157} for a generalization).
Chromatic properties of related graphs are also investigated~\cite{MR4778792,GargantiniPastineEtAl2026,MPZZ25}.

The natural question of the chromatic number of associahedra was posed by Fabila-Monroy, Flores-Pe\~naloza, Huemer, Hurtado, Wood, and Urrutia~\cite{DBLP:journals/dmtcs/MonroyFHHWU09} who showed a lower bound of 4 and an upper bound of $O(n/log n)$.
It is only recently that Addario-Berry, Reed, Scott, and Wood~\cite{Addario-Berry_Reed_Scott_Wood_2026} showed that the chromatic number of associahedra is at most $22500 \cdot \log_3 n + O(1) = O(\log n)$.
Their proof is elementary but quite involved. We give a simple proof that the chromatic
number of the $(n-1)$-dimensional associahedron is at most $6\log_2 n + 18$. This has since been improved to $O(\log \log n)$ \cite{oum2026doublelogarithmicupperboundchromatic} by Oum and Wood but we hope that the simple method here could give insights to further improvements.

\paragraph{Preliminaries.} We make the following definitions, which lead up to assigning any binary tree $T$ with $n$ nodes an integer color in the range 0 to $6 \lceil \log_2 n \rceil +17$.

\begin{itemize}
    \item If $x$ is a node in some tree $T$, let $s_T(x)$ be the number of nodes in the subtree of $x$ in $T$
    \item A node is \emph{left heavy} if its left subtree is larger than its right subtree; otherwise it is \emph{right heavy}.
    \item Let $L(T)$ be the set of left-heavy nodes in $T$.
    \item Let the \emph{left color} of $T$, $c_\ell(T)$ be $|L(T)| \bmod 3$.
    \item The \emph{heavy child} of a nonleaf node is its left child if it is left heavy and its right child if it is right heavy.
    \item Let $H(T)$ be the set of heavy children in $T$.
    \item Let the \emph{heavy child color} of $T$, $c_h(T)$, be $\left\lceil  \sum_{x \in H(T)} \log_2 s(x) \right\rceil \bmod \lceil \log_2n+1 \rceil$
    \item Let $P(T)$ be the set of \emph{heavy} paths, which is the graph obtained by connecting each non-leaf node to its heavy child. The \emph{heavy path depth} of a node is is its depth relative to the shallowest node (in the tree) on its heavy path. The \emph{path color}, $c_p(T)$, is the parity of the number of left heavy nodes of even path depth. 
    \item Finally, let the \emph{color} of $T$, $c(T)$ be $6c_h(T)+2c_\ell(T)+c_p(T)$. This is an integer in the range 0 to $6 \lceil \log_2 n + 1 \rceil +4+1 \leq 6 \log_2n+17$. It is designed so that if either the heavy child color or the left color or the path color differs between two trees then the color differs.
\end{itemize}

We now state the theorem:

\begin{theorem}
The chromatic number of the $(n-1)$-dimensional associahedron is at most $6\cdot \log_2 n + 18$.
\end{theorem}
\textbf{Proof.} Given the above scheme, all that remains to prove is that a single rotation always changes the color of the tree.

Consider a rotation that transforms a tree $T$ to $T'$. We assume the rotation is on two nodes $x$ and $y$ where $y$ is the right child of $x$ in $T$ and $x$ becomes the left child of $y$ in $T'$; the reverse case is symmetric. We consider 16 cases based on whether $x$ and $y$ are each left or right heavy before and after the rotation. These are illustrated in Figure~\ref{f1}. These are split into four groups, those where the left color changes, those where the heavy child color changes, and finally those cases that are impossible.

\begin{description}
    \item[\textcolor{\cac}{Cases II-V, VIII, IX, XII-XV (Left color changes)}] In these cases the number of left heavy nodes, $|L(T)|$ changes by 1 or 2, and thus $c_\ell(T)$ which is $|L(T)| \bmod 3$ changes.
    \item[\textcolor{\cbc}{Cases I and XVI (Heavy child color changes)}] For I (XVI is symmetric): Here the increase in $\sum_{x \in H} \log_2 s(x)$ is $\log_2 (|A|+|B|+1) - \log_2 |B| = \log_2 \frac{|A|+|B|+1}{|B|}$. 
    As $|A|>|B|$, $\log \frac{|A|+|B|+1}{|B|}\geq \log_2 \frac{|B|+|B|+1}{|B|}>1$. Also, $\log_2 \frac{|A|+|B|+1}{|B|}\leq \log_2 n$. Thus $\sum_{x \in H} \log_2 s(x)$ changes by at least 1 and at most $\log_2 n$, which causes $c_h(T) = \left\lceil  \sum_{x \in H} \log_2 s(x) \right\rceil \bmod \lceil \log_2n+1 \rceil$ to change.
    \item[\textcolor{\cdc}{Case X}] The key point of this case is that the heavy paths do not change, with the one exception that $x$ and $y$, which are adjacent on the same heavy path, have their positions change; $x$ remains right heavy and $y$ remains left heavy. Either $x$ before (right heavy) and $y$ after (left heavy) are at even depth on their path or $y$ before (left heavy) and $x$ after (right heavy) are at even depth; in either case the number of left heavy nodes of even path depth changes by one and thus the path color changes.
    \item[\textcolor{\ccc}{Cases VI, VII and XI (Impossible)}] These are impossible. For VI (XI is symmetric) the following contradictions occur:
    \begin{itemize}[topsep=0pt]
        \item $|A|>|B|+|C|+1>|B|$ and
        \item $|A|\leq |B|$
    \end{itemize}
    For VII:
      \begin{itemize}[topsep=0pt]
        \item $|A|>|B|+|C|+1>|C|$ and
        \item $|C|\geq |A|+|B|+1>|A|$
    \end{itemize}  
\end{description}

\paragraph{Acknowledgments} Jean Cardinal and John Iacono's research is supported by the F.R.S.–FNRS. Generative AI was used to create the Tikz figures but did not write a single word and was not used to obtain the result.

\newcommand{\ltree}[4]{%
\begin{tikzpicture}[
    every node/.style={font=\small},
    circle node/.style={circle, draw, inner sep=1pt,thick},
    triangle node/.style={
        regular polygon,
        regular polygon sides=3,
        draw=black,
        thick,
        inner sep=-0.5pt, 
    },
    level distance=0.7cm,
    sibling distance=0.7cm,
        edge from parent/.style={ultra thick},
    edge from parent path={
        (\tikzparentnode.south) -- (\tikzchildnode.north)
    }
]
\node[circle node] (x) {$x$}
    child {
        node[triangle node] (A) {$A$}
        edge from parent[draw=#1]
    }
    child {
        node[circle node] (y) {$y$}
        edge from parent[draw=#2]
        child {
            node[triangle node] (B) {$B$}
            edge from parent[draw=#3]
        }
        child {
            node[triangle node] (C) {$C$}
            edge from parent[draw=#4]
        }
    };
\end{tikzpicture}
}

\newcommand{\rtree}[4]{%
\begin{tikzpicture}[
    every node/.style={font=\small},
    circle node/.style={circle, draw, inner sep=1pt,thick},
    triangle node/.style={
        regular polygon,
        regular polygon sides=3,
        draw=black,
        inner sep=-0.5pt, 
        thick
    },
    level distance=0.7cm,
    sibling distance=0.7cm,
    edge from parent/.style={ultra thick},
    edge from parent path={
        (\tikzparentnode.south) -- (\tikzchildnode.north)
    }
]
\node[circle node] (x) {$y$}
    child {
        node[circle node] (y) {$x$}
        edge from parent[draw=#1]
        child {
            node[triangle node] (A) {$A$}
            edge from parent[draw=#3]
        }
        child {
            node[triangle node] (B) {$B$}
            edge from parent[draw=#4]
        }
        }
    child {
        node[triangle node] (C) {$C$}
        edge from parent[draw=#2]
    };
\end{tikzpicture}
}

\newcommand{\lta}{\ltree{red}{black}{red}{black}}
\newcommand{\ltb}{\ltree{red}{black}{black}{blue}}
\newcommand{\ltc}{\ltree{black}{blue}{red}{black}}
\newcommand{\ltd}{\ltree{black}{blue}{black}{blue}}

\newcommand{\rta}{\rtree{red}{black}{red}{black}}
\newcommand{\rtb}{\rtree{red}{black}{black}{blue}}
\newcommand{\rtc}{\rtree{black}{blue}{red}{black}}
\newcommand{\rtd}{\rtree{black}{blue}{black}{blue}}

\newcounter{casenum}

\newcommand{\ra}{%
{\stepcounter{casenum}%
\raisebox{42pt}{$\mathclap{\text{Case \Roman{casenum}}}$}}%
{\raisebox{20pt}{$\mathclap{\Rightarrow}$}}%
}

\newcommand{\lhc}[1]{Left heavy change: #1}

\begin{figure}
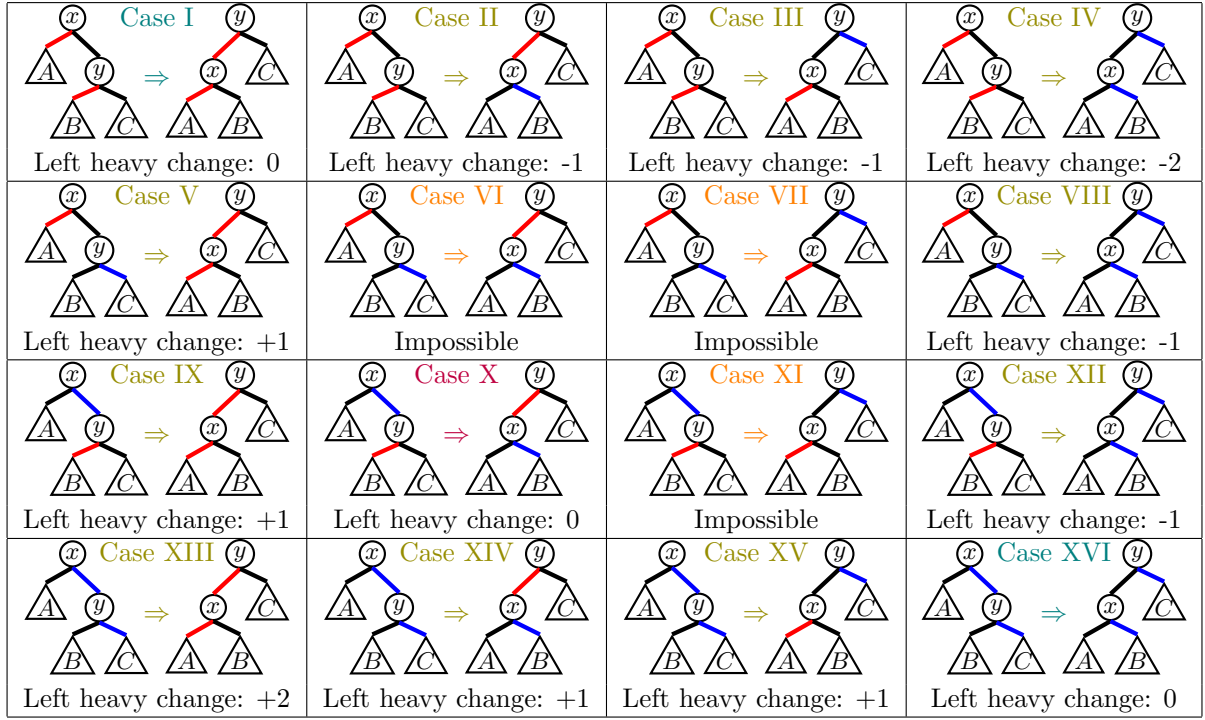

\resizebox{\textwidth}{!}{%
\begin{tabular}{|c|c|c|c|}
\hline
\lta \textcolor{\cbc}{\ra} \rta &
\lta \textcolor{\cac}{\ra}  \rtb &
\lta \textcolor{\cac}{\ra}  \rtc &
\lta \textcolor{\cac}{\ra}  \rtd 
\\
\lhc{0}
&
\lhc{-1}
&
\lhc{-1}
&
\lhc{-2}
\\ \hline
\ltb \textcolor{\cac}{\ra}  \rta &
\ltb \textcolor{\ccc}{\ra}  \rtb &
\ltb \textcolor{\ccc}{\ra}  \rtc &
\ltb \textcolor{\cac}{\ra}  \rtd 
\\
\lhc{+1}
&
Impossible
&
Impossible
&
\lhc{-1}
\\ \hline
\ltc \textcolor{\cac}{\ra}  \rta &
\ltc \textcolor{\cdc}{\ra}  \rtb &
\ltc \textcolor{\ccc}{\ra}  \rtc &
\ltc \textcolor{\cac}{\ra}  \rtd 
\\
\lhc{+1}
&
\lhc{0}
&
Impossible
&
\lhc{-1}
\\ \hline
\ltd \textcolor{\cac}{\ra} \rta &
\ltd \textcolor{\cac}{\ra}  \rtb &
\ltd \textcolor{\cac}{\ra}  \rtc &
\ltd \textcolor{\cbc}{\ra}  \rtd 
\\
\lhc{+2}
&
\lhc{+1}
&
\lhc{+1}
&
\lhc{0}
\\ \hline 
\end{tabular}
}
\caption{
The cases. We consider a left rotation on $x$ with right child $y$. A red edge indicates a left heavy node above and a blue edge indicates a right heavy node above. We consider all 16 possible cases of left and right heaviness for $x$ and $y$ before and after the rotation. 
In \textcolor{\cac}{Cases II-V, VIII, IX, XII-XV} the number of left heavy nodes change by 1 or 2 and thus the left color changes. 
In \textcolor{\cdc}{Case X} the only change in heavy paths is the relative ordering of $x$ and $y$ is swapped, each retaining their heaviness; thus the path color changes.
\textcolor{\ccc}{Cases VI, VII and XI} are impossible. 
\textcolor{\cbc}{Cases I and XVI} are possible but have no change in the number of left heavy nodes and thus require a more detailed argument and are the source of the logarithmic number of colors.
}
\label{f1}
\end{figure}

\bibliographystyle{abbrv}
\bibliography{bib}

\end{document}